\documentclass[12pt,a4paper]{article}
\usepackage{amssymb}
\usepackage{amsmath}
\usepackage{amsthm}
\usepackage{ot1patch}
\usepackage{indentfirst}

\newtheorem{tw}{Theorem}[section]

\newtheorem{lm}[tw]{Lemma}
\newtheorem{cor}[tw]{Corollary}

\newtheorem{ex}[tw]{Example}
\newtheorem{pyt}[tw]{Question}

\DeclareMathOperator{\trdeg}{tr\,deg}

\DeclareMathOperator{\jac}{Jac}
\DeclareMathOperator{\dgcd}{dgcd}

\newcommand{\arstd}{\renewcommand{\arraystretch}{1.7}}
\newcommand{\arstz}{\renewcommand{\arraystretch}{1}}

\author{Piotr J\k{e}drzejewicz\\
\normalsize Faculty of Mathematics and Computer Science\\
\normalsize Nicolaus Copernicus University\\
\normalsize Toru\'{n}, Poland}
\title{A characterization of Keller maps}

\begin{document}

\maketitle

\begin{abstract}
Let $k$ be a field of characteristic zero.
Let $\varphi$ be a $k$-en\-do\-morph\-ism
of the polynomial algebra $k[x_1,\dots,x_n]$.
It is known that $\varphi$
is an automorphism if and only if
it maps irreducible polynomials to irreducible polynomials.
In this paper we show that $\varphi$
satisfies the jacobian condition if and only if
it maps irreducible polynomials to square-free polynomials.
Therefore, the Jacobian Conjecture
is equivalent to the following statement:
every $k$-en\-do\-morph\-ism of $k[x_1,\dots,x_n]$,
mapping irreducible polynomials to square-free polynomials,
maps irreducible polynomials to irreducible polynomials.
\end{abstract}

\begin{table}[b]\footnotesize\hrule\vspace{1mm}
Keywords: jacobian conjecture, Keller map, jacobian determinant.\\
2010 Mathematics Subject Classification:
Primary 13F20, Secondary 14R15, 13N15.
\end{table}

\section{Introduction}

Throughout this article $k$ is a field of characteristic zero.
By $k[x_1,\dots,x_n]$ we denote the $k$-algebra
of polynomials in $n$ variables.
If $\varphi$ is a $k$-en\-do\-morph\-ism of $k[x_1,\dots,x_n]$,
then by $\jac \varphi$ we denote the jacobian determinant
of the polynomials $\varphi(x_1)$, $\dots$, $\varphi(x_n)$
with respect to the variables $x_1$, $\dots$, $x_n$.
If $\jac \varphi\in k\setminus\{0\}$,
then we say that $\varphi$ satisfies the jacobian condition.
In this case the respective polynomial map $F\colon k^n\to k^n$,
$F(x_1,\dots,x_n)=(\varphi(x_1),\dots,\varphi(x_n))$,
is called a Keller map.
The famous Jacobian Conjecture, stated by Keller in \cite{Keller},
asserts that every $k$-endomorphism of $k[x_1,\dots,x_n]$
satisfying the jacobian condition
is an automorphism of $k[x_1,\dots,x_n]$.

\medskip

Van den Essen and Shpilrain asked in~\cite{ES}, Problem~1,
if every $\mathbb{C}$-endo\-morph\-ism 
of $\mathbb{C}[x_1,\dots,x_n]$ mapping 
variables to variables is an automorphism
(recall that by a variable we mean an element 
of any set of $n$ generators).
The affirmative answer was given by Jelonek 
in~\cite{Jelonek}, Theorem~2. 
It was noted that this fact holds for arbitrary 
algebraically closed field of characteristic zero
(\cite{Essen}, a comment to Theorem 10.5.9 on p.~273).
Another characterization of polynomial automorphisms 
was obtained by Bakalarski in \cite{Bakalarski}, Theorem~3.7
(see also a remark at the end of \cite{Bakalarski}).
He showed that a $\mathbb{C}$-endomorphism 
of $\mathbb{C}[x_1,\dots,x_n]$ mapping 
irreducible polynomials to irreducible polynomials
is an automorphism.
We present another proof of Bakalarski's theorem
for an arbitrary field $k$ of characteristic zero
(Theorem \ref{t4}).

\medskip

The aim of this paper is to obtain a characterization 
of polynomial endomorphisms satisfying the jacobian condition
as those mapping irreducible polynomials 
to square-free polynomials (Theorem \ref{t3}).
Hence, using the result of Bakalarski,
we have the following equivalent formulation
of the Jacobian Conjecture:
every $k$-endomorphism of $k[x_1,\dots,x_n]$
mapping irreducible polynomials to square-free polynomials
maps irreducible polynomials to irreducible polynomials
(Theorem \ref{t5}).
Our characterization of $k$-endomorphisms
satisfying the jacobian condition
is an immediate consequence of Theorem \ref{t2},
where we prove that an irreducible polynomial $g$
divides the jacobian of given polynomials $f_1$, $\dots$, $f_n$
if and only if there exists
an irreducible polynomial $w\in k[x_1,\dots,x_n]$
such that $g^2$ divides $w(f_1,\dots,f_n)$.

\medskip

Basic definitions and facts are presented
in Section \ref{preliminaries}.
In Section \ref{lemmas} we prove preparatory lemmas,
which will be useful in the proof of Theorem \ref{t2},
given in Section \ref{main}.
In Section \ref{formulation} we obtain a characterization
of endomorphisms satisfying the jacobian condition
and the equivalent formulation of the Jacobian Conjecture.
Some conclusions and comments are presented
in Section \ref{remarks}.

\section{Preliminaries}

\label{preliminaries}

By a ring we mean a commutative ring with unity.
Let $A$ be a ring.
An additive map $d\colon A\to A$ such that
$d(ab)=d(a)b+ad(b)$ for $a,b\in A$
is called a derivation of $A$.
The set $A^d=\{a\in A;\;d(a)=0\}$
is called the ring of constants of $d$.
If $K$ is a subring of $A$,
then a derivation $d$ of $A$ is $K$-linear
if and only if $K\subset A^d$.
In this case we call $d$ a $K$-derivation.
Hence, if $A$ is a $k$-algebra, where $k$ is a field,
and $a_1,\dots,a_m\in A$, then a map $d\colon A\to A$
is a $k[a_1,\dots,a_m]$-derivation if and only if
$d$ is a $k$-derivation and $d(a_i)=0$ for $i=1,\dots,m$.

\medskip

If $d$ is a $k$-derivation of a $k$-algebra $A$,
where $k$ is a field,
then for an element $a\in A$ and a polynomial $w(x)\in k[x]$
we have $d(w(a))=w'(a)d(a)$.
More generally, for $a_1,\dots,a_m\in A$
and a polynomial $w(x_1,\dots,x_m)\in k[x_1,$ $\dots,x_m]$
the following holds:
$$d(w(a_1,\dots,a_m))=
\frac{\partial w}{\partial x_1}(a_1,\dots,a_m)d(a_1)+\ldots+
\frac{\partial w}{\partial x_m}(a_1,\dots,a_m)d(a_m).$$
In particular, if $d$ is a $k$-derivation of $k[x_1,\dots,x_n]$,
then $d(f)=\frac{\partial f}{\partial x_1}d(x_1)+\ldots+
\frac{\partial f}{\partial x_n}d(x_n)$
for every polynomial $f\in k[x_1,\dots,x_n]$.

\medskip

Let $k$ be a field of characteristic zero,
let $A$ be a finitely generated $k$-domain
(that is, a commutative $k$-algebra with unity,
without zero divisors)
and let $R$ be a $k$-subalgebra of $A$.
Denote by $R_0$ the field of fractions of $R$.
Nowicki (\cite{Now}, Theorem~5.4; \cite{Polder}, Theorem~4.1.4)
proved that the following conditions are equivalent:

\smallskip

\noindent
$(1)$ $R=A^d$ for some $k$-derivation $d$ of $A$,

\smallskip

\noindent
$(2)$ $R$ is integrally closed in $A$ and $R_0\cap A=R$.

\medskip

Daigle observed in "Locally nilpotent derivations"
(unpublished lecture notes, available on his website)
that the condition $(2)$ means that 
$R$ is algebraically closed in $A$ as a subring.
The present author noted in \cite{rings} that this characterization
holds also for $K$-derivations, where $K$ is a subring of $A$.
In this case we have the following corollary from
\cite{rings}, Theorem 3.1.

\begin{cor}
\label{c2}
Let $A$ be a finitely generated $K$-domain
of characteristic zero, where $K$ is a subring of $A$.
An element $b\in A$ belongs to
the ring of constants of every $K$-derivation of $A$
if and only if $b$ is algebraic over the field $K_0$.
\end{cor}

Given $n$ polynomials $f_1,\dots,f_n\in k[x_1,\dots,x_n]$,
by $\jac(f_1,\dots,f_n)$ we denote the jacobian determinant
of $f_1,\dots,f_n$ with respect to $x_1,\dots,x_n$.
Note that the map $d(f)=\jac(f_1,\dots,f_{n-1},f)$
for $f\in k[x_1,\dots,x_n]$ 
is a $k$-derivation of $k[x_1,\dots,x_n]$,
such that $d(f_i)=0$ for $i=1,\dots,n-1$.
More generally, given $m$ polynomials
$f_1,\dots,f_m\in k[x_1,\dots,x_n]$,
where $1\leqslant m\leqslant n$,
and given arbitrary $j_1,\dots,j_m\in\{1,\dots,n\}$,
by $\jac^{f_1,\dots,f_m}_{j_1,\dots,j_m}$
we denote the jacobian determinant of $f_1,\dots,f_m$
with respect to $x_{j_1},\dots,x_{j_m}$.
The map $d(f)=
\jac^{f_1,\dots,f_{i-1},f,f_{i+1},\dots,f_m}_{j_1,\dots,j_m}$
for $f\in k[x_1,\dots,x_n]$ 
is also a $k$-derivation of $k[x_1,\dots,x_n]$,
and we have $d(f_j)=0$ for $j\neq i$.

\medskip

Following \cite{jaccond}, we introduce the notion
of a differential gcd of polynomials:
$$\dgcd(f_1,\dots,f_m)=
\gcd\left(\jac^{f_1,\dots,f_m}_{j_1,\dots,j_m};\;
1\leqslant j_1,\dots,j_m\leqslant n \right)$$
for $f_1,\dots,f_m\in k[x_1,\dots,x_n]$.
Of course, $\dgcd$ is defined with respect to a scalar multiple.
For a single polynomial $f\in k[x_1,\dots,x_n]$ we have
$\dgcd(f)=c\cdot\gcd\left(\frac{\partial f}{\partial x_1},\dots,
\frac{\partial f}{\partial x_n}\right)$,
where $c\in k\setminus \{0\}$.
For $n$ polynomials $f_1,\dots,f_n\in k[x_1,\dots,x_n]$
we have $\dgcd(f_1,\dots,f_n)=c\cdot\jac(f_1,\dots,f_n)$,
where $c\in k\setminus \{0\}$.

\section{Preparatory lemmas}

\label{lemmas}

Recall that $k$ is a field of characteristic zero.
In Lemmas \ref{l1} -- \ref{l3} below we consider:
arbitrary polynomials $f_1,\dots,f_n\in k[x_1,\dots,x_n]$,
an irreducible polynomial $g\in k[x_1,\dots,x_n]$
and the factor algebra $A=k[x_1,\dots,x_n]/(g)$.
By $\overline{f}$ we denote the respective class in $A$
of a polynomial $f\in k[x_1,\dots,x_n]$, that is,
$\overline{f}=f+(g)$.

\begin{lm}
\label{l1}
For a given $i\in\{1,\dots,n\}$ consider the following condition:
$$\begin{array}{l}
\mbox{there exist $s_1,\dots,s_n\in k[x_1,\dots,x_n]$,
where $g\nmid s_i$, such that}\\
\mbox{$g\mid s_1d(f_1)+\ldots+s_nd(f_n)$
for every $k$-derivation $d$ of $k[x_1,\dots,x_n]$.}
\end{array}\leqno (\ast)$$

\medskip

\noindent
{\bf a)}
The jacobian determinant $\jac(f_1,\dots,f_n)$
is divisible by $g$ if and only if
the condition $(\ast)$ holds for some $i\in\{1,\dots,n\}$.

\medskip

\noindent
{\bf b)}
If, for a given $i\in\{1,\dots,n\}$,
the condition $(\ast)$ holds,
then $\overline{f_i}$ is algebraic over the field
$k(\,\overline{f_1},\dots,\overline{f_{i-1}},
\overline{f_{i+1}},\dots,\overline{f_n}\,)$.
\end{lm}

\begin{proof}
{\bf a)}
The jacobian determinant $\jac(f_1,\dots,f_n)$
is divisible by $g$ if and only if
the determinant of the matrix
\arstd
$$\left[\begin{array}{cccc}
\overline{\,\frac{\partial f_1}{\partial x_1}\,}&
\overline{\,\frac{\partial f_1}{\partial x_2}\,}&\cdots&
\overline{\,\frac{\partial f_1}{\partial x_n}\,}\\
\overline{\,\frac{\partial f_2}{\partial x_1}\,}&
\overline{\,\frac{\partial f_2}{\partial x_2}\,}&\cdots&
\overline{\,\frac{\partial f_2}{\partial x_n}\,}\\
\vdots&\vdots&&\vdots\\
\overline{\,\frac{\partial f_n}{\partial x_1}\,}&
\overline{\,\frac{\partial f_n}{\partial x_2}\,}&\cdots&
\overline{\,\frac{\partial f_n}{\partial x_n}\,}
\end{array}\right]$$
\arstz
equals $0$ in $A$.
If we consider this matrix over the field $A_0$,
the last condition is equivalent
to the linear dependence over $A_0$
of the rows of this matrix.
This condition can be written with coefficients in $A$:
there exist polynomials
$s_1,\dots,s_n\in k[x_1,\dots,x_n]$,
where $\overline{s_i}\neq\overline{0}$ for some $i$, such that
$$\textstyle\overline{s_1}
\left[\,\overline{\frac{\partial f_1}{\partial x_1}},\dots,
\overline{\frac{\partial f_1}{\partial x_n}}\,\right]+\ldots+
\overline{s_n}
\left[\,\overline{\frac{\partial f_n}{\partial x_1}},\dots,
\overline{\frac{\partial f_n}{\partial x_n}}\,\right]=
\left[\,\overline{0},\ldots,\overline{0}\,\right].$$
The above equality holds if and only if all the polynomials
$$\textstyle h_1=s_1\frac{\partial f_1}{\partial x_1}+\ldots+
s_n\frac{\partial f_n}{\partial x_1},
\hspace{3mm}\dots,\hspace{3mm}
h_n=s_1\frac{\partial f_1}{\partial x_n}+\ldots+
s_n\frac{\partial f_n}{\partial x_n}$$
are divisible by $g$.

\medskip

Now, observe that, for an arbitrary $k$-derivation
$d$ of $k[x_1,\dots,x_n]$, we have
\begin{eqnarray*}
\lefteqn{s_1d(f_1)+\ldots+s_nd(f_n)}\\
&=&\textstyle
s_1\left(\frac{\partial f_1}{\partial x_1}d(x_1)+\ldots+
\frac{\partial f_1}{\partial x_n}d(x_n)\right)+\ldots+
s_n\left(\frac{\partial f_n}{\partial x_1}d(x_1)+\ldots+
\frac{\partial f_n}{\partial x_n}d(x_n)\right)\\
&=&\textstyle
\left(s_1\frac{\partial f_1}{\partial x_1}+\ldots+
s_n\frac{\partial f_n}{\partial x_1}\right)d(x_1)+\ldots+
\left(s_n\frac{\partial f_1}{\partial x_n}+\ldots+
s_n\frac{\partial f_n}{\partial x_n}\right)d(x_n)\\
&=&h_1d(x_1)+\ldots+h_nd(x_n).
\end{eqnarray*}
Hence, if the polynomials $h_1,\dots,h_n$ are divisible by $g$,
then $g\mid s_1d(f_1)+\ldots+s_nd(f_n)$.
On the other hand, if the polynomial
$s_1d(f_1)+\ldots+s_nd(f_n)$ is divisible by $g$
for every $k$-derivation $d$,
then, in particular, for the partial derivatives
$d=\frac{\partial}{\partial x_j}$, $j=1,\dots,n$,
we obtain that the polynomials $h_1,\dots,h_n$
are divisible by $g$.

\medskip

\noindent
{\bf b)}
Assume that the condition $(\ast)$ holds
for some $i\in\{1,\dots,n\}$.
Let $\delta$ be an arbitrary
$k[\,\overline{f_1},\dots,\overline{f_{i-1}},
\overline{f_{i+1}},\dots,\overline{f_n}\,]$-derivation
of the factor algebra $A$,
that is, a $k$-derivation such that
$\delta(\,\overline{f_j}\,)=\overline{0}$
for each $j\in\{1,\dots,i-1,i+1,\dots,n\}$.
Consider a $k$-derivation $d$ of $k[x_1,\dots,x_n]$
such that $\delta (\,\overline{f}\,)=\overline{d(f)}$
for every $f\in k[x_1,\dots,x_n]$
(\cite{charoneel}, Lemma 3.2).
We have
$\overline{d(f_j)}=\delta (\,\overline{f_j}\,)=\overline{0}$,
that is, $g\mid d(f_j)$, for each $j\neq i$.
Hence, the condition $(\ast)$ yields that $g\mid s_id(f_i)$,
so $g\mid d(f_i)$, because $g\nmid s_i$.
This means that
$\delta (\,\overline{f_i}\,)=\overline{d(f_i)}=\overline{0}$.
By Corollary \ref{c2},
since $\delta (\,\overline{f_i}\,)=\overline{0}$
for an arbitrary
$k[\,\overline{f_1},\dots,\overline{f_{i-1}},
\overline{f_{i+1}},\dots,\overline{f_n}\,]$-derivation
$\delta$ of $A$,
$\overline{f_i}$ is algebraic over the field
$k(\,\overline{f_1},\dots,\overline{f_{i-1}},
\overline{f_{i+1}},\dots,\overline{f_n}\,)$.
\end{proof}

Note the following easy observation.

\begin{lm}
\label{l2}
Let $m\in\{1,\dots,n\}$.

\medskip

\noindent
{\bf a)}
The elements $\overline{f_1},\dots,\overline{f_m}\in A$
are algebraically dependent over $k$
if and only if $g\mid w(f_1,\dots,f_m)$
for some nonzero polynomial $w\in k[x_1,\dots,x_m]$.

\medskip

\noindent
{\bf b)}
Let $i\in\{1,\dots,m\}$.
The element $\overline{f_i}\in A$
is algebraic over the field
$k(\,\overline{f_1},\dots,\overline{f_{i-1}},
\overline{f_{i+1}},\dots,\overline{f_n}\,)$
if and only if $g\mid w(f_1,\dots,f_m)$
for some nonzero polynomial $w\in k[x_1,\dots,x_m]$
of positive degree with respect to $x_i$.
\end{lm}

In the case of $n$ polynomials in $n$ variables
we have the following.

\begin{lm}
\label{l3}
There exists an irreducible polynomial
$w\in k[x_1,\dots,x_n]$ such that $g\mid w(f_1,\dots,f_n)$.
\end{lm}

\begin{proof}
The ideal $(g)$ has height $1$,
so the Krull dimension of $A$ equals $n-1$.
Hence, the elements $\overline{f_1},\dots,\overline{f_n}$
are algebraically dependent over $k$.
Then, by Lemma~\ref{l2}.a,
there exists a nonzero polynomial
$u\in k[x_1,\dots,x_n]$ such that $g\mid u(f_1,\dots,f_n)$.
The polynomial $u$ is obviously non-constant.
Then, for some irreducible factor $w$ of $u$,
the polynomial $w(f_1,\dots,f_n)$ is divisible by~$g$.
\end{proof}

\begin{lm}
\label{l4}
Assume that $n\geqslant 2$ and $0\leqslant r\leqslant n-2$.
Consider polynomials
$u_1\in k[x_1,\dots,x_r,x_{r+1}]\setminus k[x_1,\dots,x_r]$
and $u_2\in k[x_1,\dots,x_r,x_{r+2}]\setminus k[x_1,\dots,x_r]$.
If the degrees of $u_1$ with respect to $x_{r+1}$
and of $u_2$ with respect to $x_{r+2}$ are relatively prime,
then there exist nonzero polynomials $w_1\in k[x_1,\dots,x_r]$,
$w_2\in k[x_1,\dots,x_r,x_{r+1},x_{r+2}]\setminus k[x_1,\dots,x_r]$
such that $w_2$ is irreducible and $u_1+u_2=w_1w_2$.
\end{lm}

\begin{proof}
Consider a decomposition of $u_1+u_2$
into irreducible factors in $k[x_1,$ $\dots,x_r,x_{r+1},x_{r+2}]$:
$u_1+u_2=v_1\ldots v_sv_{s+1}\ldots v_t$,
where $v_1,\dots,v_s\in k[x_1,\dots,$ $x_r]$
and $v_{s+1},\dots,v_t\not\in k[x_1,\dots,x_r]$,
$0\leqslant s\leqslant t$.
Observe that $s<t$, because $u_1+u_2\not\in k[x_1,\dots,x_r]$.

\medskip

Now, consider the field $L=k(x_1,\dots,x_r)$.
Since the degrees of the polynomials:
$u_1$ in $L[x_{r+1}]$ and $u_2$ in $L[x_{r+2}]$
are positive and relatively prime,
the polynomial $u_1+u_2$ is irreducible
in $L[x_{r+1},x_{r+2}]$,
by Corollary~3 to Theorem~21 in \cite{Schinzel}, p.~94
(see also \cite{Ehrenfeucht}).
Hence $t=s+1$.
Finally, put $w_1=v_1\ldots v_s$ ($w_1=1$ if $s=0$)
and $w_2=v_{s+1}$.
\end{proof}

\begin{lm}
\label{l5}
Let $w\in k[x_1,\dots,x_m]$ be an irreducible polynomial
such that $\frac{\partial w}{\partial x_i}\neq 0$
for some $i\in\{1,\dots,m\}$.
Then there exist polynomials $v_1,v_2\in k[x_1,\dots,x_m]$
and $v\in k[x_1,\dots,x_{i-1},x_{i+1},\dots,x_m]\setminus\{0\}$
such that 
$$\textstyle v_1w+v_2\frac{\partial w}{\partial x_i}=v.$$
\end{lm}

\begin{proof}
Consider the field $L=k(x_1,\dots,x_{i-1},x_{i+1},\dots,x_m)$.
The polynomial $w$ is irreducible in $L[x_i]$,
and the polynomial $\frac{\partial w}{\partial x_i}$
is nonzero, so they are relatively prime in $L[x_i]$.
Hence there exist polynomials $u_1,u_2\in L[x_i]$ such that
$$\textstyle u_1w+u_2\frac{\partial w}{\partial x_i}=1.$$
Let $v$ ($v\in k[x_1,\dots,x_{i-1},x_{i+1},\dots,x_m]$)
be the least common denominator of the coefficients 
of polynomials $u_1$, $u_2$. 
Multiplying the above equality by $v$ 
and denoting $v_1=u_1v$, $v_2=u_2v$ we get the lemma.
\end{proof}

\section{Irreducible factors of jacobians}

\label{main}

\begin{tw}
\label{t2}
Let $k$ be a field of characteristic zero,
let $f_1,\dots,f_n\in k[x_1,$ $\dots,x_n]$ be arbitrary polynomials,
and let $g\in k[x_1,\dots,x_n]$ be an irreducible polynomial.
The following conditions are equivalent:

\medskip

\noindent
$(i)$ \ 
$g$ divides $\jac(f_1,\dots,f_n)$,

\medskip

\noindent
$(ii)$ \ 
$g^2$ divides $w(f_1,\dots,f_n)$ for some irreducible polynomial
$w\in k[x_1,\dots,x_n]$.
\end{tw}

\begin{proof}
$(i)\Rightarrow (ii)$ \
Assume that $g\mid\jac(f_1,\dots,f_n)$.

\medskip

By Lemma~\ref{l3}, $g\mid w(f_1,\dots,f_n)$
for some irreducible polynomial $w\in k[x_1,\dots,x_n]$,
so $$w(f_1,\dots,f_n)=gh\leqno (1)$$
for some $h\in k[x_1,\dots,x_n]$.
Without loss of generality, we may assume
that $w$ is of positive degree with respect to $x_n$,
so $\overline{f_n}$ is algebraic over the field
$k(\,\overline{f_1},\dots,\overline{f_{n-1}}\,)$,
by Lemma~\ref{l2}.b.
Assume that $g^2\nmid w(f_1,\dots,f_n)$, that is, $g\nmid h$.

\medskip

Consider the $k$-derivation $d_n$ of $k[x_1,\dots,x_n]$ defined by
$$d_n(f)=\jac(f_1,\dots,f_{n-1},f)$$ for $f\in k[x_1,\dots,x_n]$.
Observe that $d_n(f_i)=0$ for $i=1,\dots,n-1$
and $d_n(f_n)=\jac(f_1,\dots,f_n)$.
Applying the derivation $d_n$ to both sides of $(1)$ we obtain
$$\textstyle
\frac{\partial w}{\partial x_n}(f_1,\dots,f_n)d_n(f_n)=
d_n(g)h+gd_n(h).$$
Since $g\mid d_n(f_n)$ and $g\nmid h$,
we have $g\mid d_n(g)$, that is, $g\mid\jac(f_1,\dots,f_{n-1},g)$.

\medskip

From Lemma \ref{l1} we obtain that there exist polynomials
$s_1,\dots,s_n\in$ $k[x_1,\dots,x_n]$,
where $g\nmid s_i$ for some $i\in\{1,\dots,n\}$, such that
$$g\mid s_1d(f_1)+\ldots+s_{n-1}d(f_{n-1})+s_nd(g)$$
for every $k$-derivation $d$ of $k[x_1,\dots,x_n]$.
Note that the polynomials $s_1,\dots,$ $s_{n-1}$
can not all together be divisible by $g$.
Indeed, in this case we would have $g\nmid s_n$ and $g\mid s_nd(g)$,
so $g\mid d(g)$ for every $k$-derivation $d$,
what is not true for $d=\frac{\partial}{\partial x_j}$
such that $\frac{\partial g}{\partial x_j}\neq 0$.

\medskip

Thus $g\nmid s_i$ for some $i\in\{1,\dots,n-1\}$;
we may assume that $g\nmid s_{n-1}$.
By Lemma \ref{l1}, $\overline{f_{n-1}}$ is algebraic over the field
$k(\,\overline{f_1},\dots,\overline{f_{n-2}},\overline{g}\,)=
k(\,\overline{f_1},\dots,$ $\overline{f_{n-2}}\,)$.
Recall that $\overline{f_n}$ is algebraic over
$k(\,\overline{f_1},\dots,\overline{f_{n-1}}\,)$,
so if we denote
$r=\trdeg_kk(\,\overline{f_1},\dots,\overline{f_n}\,)$,
we have $r\leqslant n-2$.
Hence, we may assume that
$\overline{f_1},\dots,\overline{f_r}$
are algebraically independent over $k$.

\medskip

Let $L=k(\,\overline{f_1},\dots,\overline{f_r}\,)$.
Since $\overline{f_{r+1}}$ and $\overline{f_{r+2}}$
are algebraic over $L$,
there exist nonzero polynomials
$v_1\in k[x_1,\dots,x_r,x_{r+1}]$,
$v_2\in k[x_1,\dots,x_r,x_{r+2}]$
of positive degrees $t_1$, $t_2$
with respect to $x_{r+1}$, $x_{r+2}$, respectively,
such that the polynomials
$v_1(f_1,\dots,f_r,f_{r+1})$ and $v_2(f_1,\dots,f_r,f_{r+2})$
are both divisible by $g$
(Lemma~\ref{l2}.b).
Put $u_1=v_1^2x_{r+1}^{2t_2+1}$, $u_2=v_2^2x_{r+2}^{2t_1}$.
Then the polynomials
$u_1(f_1,\dots,f_r,f_{r+1})$ and $u_2(f_1,\dots,f_r,f_{r+2})$
are both divisible by $g^2$ and, by Lemma \ref{l4},
there exist nonzero polynomials $w_1\in k[x_1,\dots,x_r]$,
$w_2\in k[x_1,\dots,x_r,x_{r+1},x_{r+2}]$ such that
$w_2$ is irreducible and $u_1+u_2=w_1w_2$.
We obtain that
$$g^2\mid w_1(f_1,\dots,f_r)w_2(f_1,\dots,f_r,f_{r+1},f_{r+2}),$$
but $g\nmid w_1(f_1,\dots,f_r)$ by Lemma~\ref{l2}.a, 
because $\overline{f_1},\dots,\overline{f_r}$
are algebraically independent over $k$.
Finally, $g^2\mid w_2(f_1,\dots,f_r,f_{r+1},f_{r+2})$.

\medskip

$(ii)\Rightarrow (i)$ \
We will show by induction on $m\in\{1,\dots,n\}$ that
for $m$ arbitrary polynomials $f_1,\dots,f_m\in k[x_1,\dots,x_n]$
and an irreducible polynomial $g\in k[x_1,\dots,x_n]$,
if $g^2\mid w(f_1,\dots,f_m)$
for some irreducible polynomial $w\in k[x_1,\dots,x_m]$,
then $g\mid \dgcd(f_1,\dots,f_m)$.
Recall that $\dgcd(f_1,\dots,f_m)=
\gcd\left(\jac^{f_1,\dots,f_m}_{j_1,\dots,j_m};\;
1\leqslant j_1,\dots,j_m\leqslant n \right)$.

\medskip

Let $m=1$.
Assume that $g^2\mid w(f_1)$,
where $w\in k[x_1]$ is an irreducible polynomial,
so $w(f_1)=g^2h$ for some $h\in k[x_1,\dots,x_n]$.
Applying the partial derivative with respect to $x_i$
for $i\in\{1,\dots,n\}$ we obtain
$w'(f_1)\frac{\partial f_1}{\partial x_i}=
2g\frac{\partial g}{\partial x_i}h+
g^2\frac{\partial h}{\partial x_i}$,
so $g\mid w'(f_1)\frac{\partial f_1}{\partial x_i}$.
Since $w$ is irreducible,
$w'$ is relatively prime to $w$,
so $uw+vw'=1$ for some polynomials $u,v\in k[x_1]$.
This yields $u(f_1)w(f_1)+v(f_1)w'(f_1)=1$,
so $g\nmid w'(f_1)$.
Therefore $g\mid \frac{\partial f_1}{\partial x_i}$
for each $i$, so $g\mid\dgcd(f_1)$.

\medskip

Now, let $m\in\{2,\dots,n\}$.
Assume that the induction hypothesis holds for $m-1$.
Assume that $g^2\mid w(f_1,\dots,f_m)$
for some irreducible polynomial $w\in k[x_1,\dots,x_m]$:
$$w(f_1,\dots,f_m)=g^2h, \leqno (2)$$
where $h\in k[x_1,\dots,x_n]$.

\medskip

First, consider the case when
$g^2\mid u(f_1,\dots,f_{i-1},f_{i+1},\dots,f_m)$
for some $i\in\{1,\dots,m\}$
and some irreducible polynomial
$u\in k[x_1,\dots,x_{i-1},x_{i+1},\dots,x_m]$.
In this case, by the induction hypothesis,
$g\mid\dgcd(f_1,\dots,f_{i-1},f_{i+1},\dots,f_m)$,
so every jacobian determinant of
$f_1,\dots,f_{i-1},f_{i+1},\dots,f_m$ is divisible by $g$.
Then, for arbitrary $j_1,\dots,j_m\in\{1,\dots,n\}$,
from the Laplace expansion with respect to $i$-th row,
we see that the determinant
$\jac^{f_1,\dots,f_m}_{j_1,\dots,j_m}$ is divisible by $g$,
so $g\mid\dgcd(f_1,\dots,f_m)$.

\medskip

Now, assume that
$g^2\nmid u(f_1,\dots,f_{i-1},f_{i+1},\dots,f_m)$
for each $i\in\{1,\dots,m\}$
and every irreducible polynomial
$u\in k[x_1,\dots,x_{i-1},x_{i+1},\dots,x_m]$.
Hence, in particular,
$\frac{\partial w}{\partial x_i}\neq 0$.
Suppose that $g\nmid\dgcd(f_1,\dots,f_m)$, that is,
$g\nmid\jac^{f_1,\dots,f_m}_{j_1,\dots,j_m}$
for some $j_1,\dots,j_m\in\{1,\dots,n\}$.
Denote by $d_i$, for $i=1,\dots,m$,
the $k$-derivation of $k[x_1,\dots,x_n]$ defined by
$$d_i(f)=
\jac^{f_1,\dots,f_{i-1},f,f_{i+1},\dots,f_m}_{j_1,\dots,j_m}$$
for $f\in k[x_1,\dots,x_n]$.
Observe that $d_i(f_j)=0$ for $j\neq i$
and $d_i(f_i)=\jac^{f_1,\dots,f_m}_{j_1,\dots,j_m}$.
Applying the derivation $d_i$ to both sides of $(2)$
we have
$$\textstyle \frac{\partial w}{\partial x_i}(f_1,\dots,f_m)
\jac^{f_1,\dots,f_m}_{j_1,\dots,j_m}=
2gd_i(g)h+g^2d_i(h),$$
so $g\mid\frac{\partial w}{\partial x_i}(f_1,\dots,f_m)$.

\medskip

From Lemma \ref{l5} we obtain that
$g\mid v(f_1,\dots,f_{i-1},f_{i+1},\dots,f_m)$
for some nonzero polynomial
$v\in k[x_1,\dots,x_{i-1},x_{i+1},\dots,x_m]$.
The polynomial $v$ is obviously non-constant.
Then there exists an irreducible polynomial
$u_i\in k[x_1,\dots,x_{i-1},x_{i+1},\dots,x_m]$
such that the polynomial
$u_i(f_1,\dots,f_{i-1},f_{i+1},\dots,$ $f_m)$
is divisible by $g$, that is,
$$u_i(f_1,\dots,f_{i-1},f_{i+1},\dots,f_m)=gs_i\leqno (3)$$
for some $s_i\in k[x_1,\dots,x_n]$.
By the assumption,
the left side of $(3)$ is not divisible by $g^2$,
so $g\nmid s_i$.
Applying the derivation $d_i$ to both sides of $(3)$
we obtain $0=d_i(g)s_i+gd_i(s_i)$,
so $g\mid d_i(g)$ (for arbitrary $i\in\{1,\dots,m\}$).

\medskip

Now, consider arbitrary $j\in\{1,\dots,m\}$
and apply the derivation $d_j$ to both sides of $(3)$
for $i\in\{1,\dots,m\}$, $i\neq j$:
$$\textstyle \frac{\partial u_i}{\partial x_j}
(f_1,\dots,f_{i-1},f_{i+1},\dots,f_m)
\jac^{f_1,\dots,f_m}_{j_1,\dots,j_m}=
d_j(g)s_i+gd_j(s_i).$$
Since $g\mid d_j(g)$
and $g\nmid\jac^{f_1,\dots,f_m}_{j_1,\dots,j_m}$,
we have
$g\mid\frac{\partial u_i}{\partial x_j}
(f_1,\dots,f_{i-1},f_{i+1},\dots,f_m)$.
On the other hand,
applying the derivation $\frac{\partial}{\partial x_j}$
to both sides of $(3)$ we obtain
$$\begin{array}{l}
\textstyle
\frac{\partial u_i}{\partial x_1}
(f_1,\dots,f_{i-1},f_{i+1},\dots,f_m)
\frac{\partial f_1}{\partial x_j}+\ldots+
\frac{\partial u_i}{\partial x_m}
(f_1,\dots,f_{i-1},f_{i+1},\dots,f_m)
\frac{\partial f_m}{\partial x_j}\\
\textstyle \;\;\;\;=\;\;
\frac{\partial g}{\partial x_j}s_i+
g\frac{\partial s_i}{\partial x_j}.
\end{array}$$
Recall that $g\nmid s_i$,
so $g\mid\frac{\partial g}{\partial x_j}$,
that is, $\frac{\partial g}{\partial x_j}=0$
for each $j\in\{1,\dots,m\}$,
a contradiction.
\end{proof}

Note the following immediate consequence of Theorem~\ref{t2}.

\begin{cor}
\label{c1}
For arbitrary polynomials $f_1,\dots,f_n\in k[x_1,\dots,x_n]$
the following conditions are equivalent:

\smallskip

\noindent
$(i)$ \
$\jac(f_1,\dots,f_n)\in k\setminus\{0\}$,

\smallskip

\noindent
$(ii)$ \
for every irreducible polynomial $w\in k[x_1,\dots,x_n]$
the polynomial $w(f_1,$ $\dots,f_n)$ is square-free.
\end{cor}

\section{An equivalent formulation of the Jacobian Conjecture}

\label{formulation}

If $\varphi$ is a $k$-endomorphism of $k[x_1,\dots,x_n]$,
then by $\jac \varphi$ we denote
the jacobian determinant of the polynomials
$\varphi(x_1)$, $\dots$, $\varphi(x_n)$
with respect to $x_1$, $\dots$, $x_n$:
$$\jac \varphi=\jac(\varphi(x_1),\dots,\varphi(x_n)).$$
We obtain the following characterization
of $k$-endomorphisms satisfying the jacobian condition.

\begin{tw}
\label{t3}
Let $k$ be a field of characteristic zero.
Let $\varphi$ be a $k$-en\-do\-morph\-ism
of the polynomial algebra $k[x_1,\dots,x_n]$.
The following conditions are equivalent:

\smallskip

\noindent
$(i)$ \
$\jac \varphi\in k\setminus\{0\}$,

\smallskip

\noindent
$(ii)$ \
for every irreducible polynomial $w\in k[x_1,\dots,x_n]$
the polynomial $\varphi(w)$ is square-free.
\end{tw}

\begin{proof}
Put $f_1=\varphi(x_1)$, $\dots$, $f_n=\varphi(x_n)$.
Since $\varphi$ is a $k$-endomorphism,
for every polynomial $w\in k[x_1,\dots,x_n]$ we have
$$\varphi(w(x_1,\dots,x_n))=
w(\varphi(x_1),\dots,\varphi(x_n))=
w(f_1,\dots,f_n).$$
The rest follows from Corollary \ref{c1}.
\end{proof}

The following theorem was obtained
by Bakalarski in \cite{Bakalarski} (Theorem 3.7)
under an additional assumption,
but it was noted in a remark added in the proof
that this assumption is not necessary.
Here we present another proof of Bakalarski's theorem,
based on our Lemma~\ref{l3}.

\begin{tw}[Bakalarski]
\label{t4}
Let $k$ be a field of characteristic zero
and let $\varphi$ be a $k$-endomorphism
of the polynomial algebra $k[x_1,\dots,x_n]$.
The following conditions are equivalent:

\smallskip

\noindent
$(i)$ \
$\varphi$ is a $k$-automorphism of $k[x_1,\dots,x_n]$,

\smallskip

\noindent
$(ii)$ \
for every irreducible polynomial $w\in k[x_1,\dots,x_n]$
the polynomial $\varphi(w)$ is irreducible.
\end{tw}

\begin{proof}
Every automorphism of a ring maps irreducible elements
into irreducible elements,
so it is enough to prove the implication $(ii)\Rightarrow (i)$.
Assume that $\varphi(w)$ is an irreducible polynomial
for every irreducible $w\in k[x_1,\dots,x_n]$.
Observe that $\varphi$ is a monomorphism:
if $\varphi(f)=0$ and $f=g_1\ldots g_r$
is a decomposition into irreducible factors,
then $\varphi(g_i)=0$ for some $i$, contrary to the assumption.

\medskip

Now we will prove that $\varphi$ is surjective.
Put $f_i=\varphi(x_i)$ for $i=1,\dots,n$.
Suppose that there exists a polynomial $g\in k[x_1,\dots,x_n]$,
such that $g\not\in k[f_1,\dots,f_n]$.
In this case at least one of irreducible factors of $g$
does not belong to $k[f_1,\dots,f_n]$,
so we may assume that $g$ is irreducible.
Then, by Lemma \ref{l3}, $g\mid w(f_1,\dots,f_n)$
for some irreducible polynomial $w\in k[x_1,\dots,x_n]$,
that is, $w(f_1,\dots,f_n)=gh$, where $h\in k[x_1,\dots,x_n]$.
However, $h\not\in k$, because $g\not\in k[f_1,\dots,f_n]$,
so $w(f_1,\dots,f_n)$ is a reducible polynomial.
\end{proof}

By Theorems \ref{t3} and \ref{t4} we have.

\begin{tw}
\label{t5}
Let $k$ be a field of characteristic zero,
let $n$ be a positive integer.
The following conditions are equivalent:

\smallskip

\noindent
$(i)$ \
every $k$-endomorphism $\varphi$ of $k[x_1,\dots,x_n]$
such that $\jac \varphi\in k\setminus\{0\}$
is an automorphism of $k[x_1,\dots,x_n]$
(the Jacobian Conjecture),

\smallskip

\noindent
$(ii)$ \
every $k$-endomorphism of $k[x_1,\dots,x_n]$
mapping irreducible polynomials to square-free polynomials
maps irreducible polynomials to irreducible polynomials.
\end{tw}

\section{Final remarks}

\label{remarks}

\noindent
{\bf Remark 1.}
From Theorem \ref{t3} we know that
a $k$-endomorphism $\varphi$ of $k[x_1,\dots,x_n]$
satisfies the jacobian condition if and only if
it maps irreducible polynomials to square-free polynomials.
It is natural to ask if there exists
a non-trivial example of such a $k$-endomorphism;
non-trivial in the following sense: 
for some irreducible polynomial $w\in k[x_1,\dots,x_n]$ 
the polynomial $\varphi(w)$ is reducible.
From Theorem \ref{t5} we know that such an example
would be a counter-example to the Jacobian Conjecture,
and if such an example does not exist,
the Jacobian Conjecture is true.

\bigskip

\noindent
{\bf Remark 2.}
It may be interesting to consider the following property
of a given ring (a~commutative ring with unity):
$$\begin{array}{l}
\mbox{\em every endomorphism mapping irreducible elements
to square-free}\\
\mbox{\em elements maps irreducible elements to irreducible
elements.}
\end{array}\leqno (\ast)$$

\begin{pyt}
\label{q1}
Let $R$ be a unique factorization domain
satisfying the condition~$(\ast)$.
Does the ring $R[x]$
of polynomials in one variable over $R$
also satisfy the condition~$(\ast)$?
\end{pyt}

If the answer to this question is positive,
then the Jacobian Conjecture is true.
Namely, in this case, by an obvious induction,
every ring endomorphism of $k[x_1,\dots,x_n]$
mapping irreducible polynomials to square-free polynomials
maps irreducible polynomials to irreducible polynomials.
And then, in particular,
every $k$-endomorphism of $k[x_1,\dots,x_n]$
mapping irreducible polynomials to square-free polynomials
maps irreducible polynomials to irreducible polynomials.

\bigskip

\noindent
{\bf Remark 3.}
Theorem \ref{t2} is a multi-dimensional version
of the following lemma of Freudenburg (\cite{Freudenburg}):
if an irreducible polynomial $g\in\mathbb{C}[x,y]$
divides both partial derivatives
$\frac{\partial f}{\partial x}$, $\frac{\partial f}{\partial y}$
of a given polynomial $f\in\mathbb{C}[x,y]$,
then $g$ divides $f+c$ for some $c\in\mathbb{C}$.
Van den Essen, Nowicki and Tyc (\cite{ENT})
generalized this lemma for $n$ variables
over an algebraically closed field of characteristic $0$.
In \cite{charoneel} the author obtained the following
generalization for an arbitrary field $k$ of characteristic $0$
(not necessarily algebraically closed):
an irreducible polynomial $g\in k[x_1,\dots,x_n]$
divides all partial derivatives
$\frac{\partial f}{\partial x_1}$, $\dots$,
$\frac{\partial f}{\partial x_n}$
of a given polynomial $f\in k[x_1,\dots,x_n]$
if and only if $g^2$ divides $W(f)$
for some irreducible polynomial $W(T)\in k[T]$.

\medskip

Let us take a closer look at a very specific analogy between
the cases of a single polynomial and of $n$ polynomials.
In fact, comparing the proofs, we may argue
that there is no real analogy here.
The only crucial implication
for a single polynomial $f\in k[x_1,\dots,x_n]$
is the following:
if an irreducible polynomial $g\in k[x_1,\dots,x_n]$
divides $\frac{\partial f}{\partial x_1}$, $\dots$,
$\frac{\partial f}{\partial x_n}$,
then $g$ divides $W(f)$
for some irreducible polynomial $W(T)\in k[T]$.
For $n$ polynomials $f_1,\dots,f_n\in k[x_1,\dots,x_n]$
and an irreducible polynomial $g\in k[x_1,\dots,x_n]$,
without any assumptions, there always exists
an irreducible polynomial $w\in k[x_1,\dots,x_n]$
such that $g$ divides $w(f_1,\dots,f_n)$.
We have established this fact in Lemma~\ref{l3}.

\medskip

Now, for a single polynomial $f\in k[x_1,\dots,x_n]$
it is easy to show that
if an irreducible polynomial $g\in k[x_1,\dots,x_n]$
divides $\frac{\partial f}{\partial x_1}$, $\dots$,
$\frac{\partial f}{\partial x_n}$ and $W(f)$,
where $W(T)\in k[T]$ is an irreducible polynomial,
then $g^2$ divides $W(f)$.
The analog of this fact for $n$ polynomials is,
in general, not true, as the following example shows.

\begin{ex}[Gwo\'{z}dziewicz, Jelonek]
\label{e1}
Consider the following polynomials in $k[x,y]$:
$f_1=x$, $f_2=xy$, $g=x$ and $w=x$.
Then $\jac(f_1,f_2)=x$ and $w(f_1,f_2)=x$
are divisible by $g$,
but $w(f_1,f_2)$ is not divisible by $g^2$.
\end{ex}

However, we still can prove that if $g$
divides the jacobian of $f_1$, $\dots$, $f_n$,
then there exists
an irreducible polynomial $w\in k[x_1,\dots,x_n]$
such that $g^2$ divides $w(f_1,\dots,f_n)$.
Finally, the reverse implication
is also not easy to be proved,
in contrast to the case of a single polynomial.
It is an easy exercise to show for
$f,g\in k[x_1,\dots,x_n]$, where $g$ is irreducible,
that if $g^2$ divides $W(f)$
for some irreducible polynomial $W(T)\in k[T]$,
then $g$ divides $\frac{\partial f}{\partial x_i}$
for $i=1,\dots,n$.

\bigskip

{\bf Acknowledgements.}
The author would like to thank 
Dr.\ Janusz Gwo\'{z}\-dzie\-wicz 
and Prof.\ Zbigniew Jelonek for Example~\ref{e1}.
The author would also like to thank 
Prof.\ Ludwik M.\ Dru\.{z}kowski
for the correction of Lemma~\ref{l5} 
from the first version of this paper.

\end{document}